\documentclass[pdflatex,sn-mathphys-num,iicol]{sn-jnl}

\usepackage{graphicx}%
\usepackage{multirow}%
\usepackage{amsmath,amssymb,amsfonts}%
\usepackage{amsthm}%
\usepackage{mathrsfs}%
\usepackage[title]{appendix}%
\usepackage{xcolor}%
\usepackage{textcomp}%
\usepackage{manyfoot}%
\usepackage{booktabs}%
\usepackage{algorithm}%
\usepackage{algorithmicx}%
\usepackage{listings}%
\usepackage[T1]{fontenc}      
\usepackage{hyperref}         
\usepackage{url}              
\usepackage{booktabs}         
\usepackage{amsfonts}         
\usepackage{nicefrac}         
\usepackage{microtype}        
\usepackage{subcaption}       
\usepackage{doi}              
\usepackage{float}            
\usepackage{mathtools}        
\usepackage{amsmath, amssymb, epsfig, array} 
\usepackage{tcolorbox}        
\usepackage{algorithm}        
\usepackage[noend]{algpseudocode} 
\usepackage{cleveref}
\usepackage{cases}
\usepackage{cuted}
\usepackage{booktabs}
\usepackage{tikz} 

\newtheorem{theorem}{Theorem}
\newtheorem{remark}{Remark}%
\newtheorem{assumption}{Assumption}
\newtheorem{lemma}{Lemma}

\newtheorem{definition}{Definition}%

\usetikzlibrary{spy}
\definecolor{spycolor}{RGB}{235, 135, 35}

\newcommand{\figDoubleColumn}[5]{
\begin{figure*}[!t]
\begin{center}
  \includegraphics[trim=#2, clip, width=#1\linewidth]{./figures/#3}
  \caption{#5}
  \label{figure:#4}
\end{center}
\end{figure*}
}

\usepackage{MnSymbol}
\definecolor{indigo}{RGB}{75, 0, 130}
\newcommand{\indigoCircle}{\textcolor{indigo}{\Large$\bullet$}}
\newcommand{\indigoStar}{\textcolor{indigo}{\Large$\filledstar$}}
\newcommand{\indigoTriangle}{\textcolor{indigo}{\Large$\filledtriangleup$}}

\newcommand{\rv}[2][final]{%
  \ifthenelse{\equal{#1}{revision}}%
    {\textcolor{blue}{#2}}
    {#2}
}
\newcommand{\rvv}[2][final]{%
  \ifthenelse{\equal{#1}{revision}}%
    {\textcolor{blue}{#2}}
    {#2}
}

\begin{document}

\title[Adaptively Inexact method for Bilevel Learning]{An Adaptively Inexact Method for Bilevel Learning Using Primal-Dual Style Differentiation}


\author[1]{\fnm{Lea} \sur{Bogensperger}}\email{lea.bogensperger@uzh.ch}

\author[2]{\fnm{Matthias J.} \sur{Ehrhardt}}\email{m.ehrhardt@bath.ac.uk}

\author[3]{\fnm{Thomas} \sur{Pock}} \email{thomas.pock@tugraz.at}

\author*[2]{\fnm{Mohammad Sadegh} \sur{Salehi}}\email{mss226@bath.ac.uk}

\author[2]{\fnm{Hok Shing} \sur{Wong}}\email{hsw43@bath.ac.uk}
\affil[1]{\orgdiv{Department of Quantitative Biomedicine}, \orgname{University of Zurich}, \orgaddress{\postcode{8057}, \city{Zurich}, \country{Switzerland}}}

\affil[2]{\orgdiv{Department of Mathematical Sciences}, \orgname{University of Bath}, \orgaddress{\postcode{BA2 7AY}, \city{Bath}, \country{UK}}}

\affil[3]{\orgdiv{Institute of Computer Graphics and Vision}, \orgname{Graz University of Technology}, \orgaddress{\postcode{8010}, \city{Graz}, \country{Austria}}}

\abstract{We consider a bilevel learning framework for learning linear operators. In this framework, the learnable parameters are optimized via a loss function that also depends on the minimizer of a convex optimization problem (denoted lower-level problem). We utilize an iterative algorithm called `piggyback' to compute the gradient of the loss and minimizer of the lower-level problem. Given that the lower-level problem is solved numerically, the loss function and thus its gradient can only be computed inexactly. To estimate the accuracy of the computed hypergradient, we derive an a-posteriori error bound, which provides guides for setting the tolerance for the lower-level problem, as well as the piggyback algorithm. To efficiently solve the upper-level optimization, we also propose an adaptive method for choosing a suitable step-size. To illustrate the proposed method, we consider a few learned regularizer problems, such as training an input-convex neural network.}




\maketitle

\section{Introduction}
\label{sec:intro}
Variational image reconstruction has proven to be a successful approach for solving inverse problems in imaging such as image reconstruction. The reconstruction can be obtained by solving the following minimization problem:
\begin{equation}
    \min_xD(Ax,u)+ R_{\theta}(x).
\end{equation}
Here $D$ denotes the data fidelity, $R_\theta$ is the regularizer, which is parameterized by $\theta$. Moreover, $A$ is a linear operator related to the specific inverse problem of interest, such as image denoising, deblurring or tomography, to name but a few.


Classical approaches rely on hand-crafted regularizers designed to encode prior information about the desired reconstruction. Notable examples include total variation (TV) \cite{rudin1992nonlinear}, total generalized variation (TGV) \cite{bredies2010total}, and sparsity-promoting regularizer \cite{daubechies2004iterative}. In these methods, the parameter 
$\theta$ often corresponds to the regularization parameter, controlling the trade-off between data fidelity and regularization. Various parameter selection strategies have been proposed in the literature, including heuristic approaches like the L-curve method \cite{hansen1992analysis,hansen2010discrete}, the discrepancy principle and other a-posteriori rules \cite{benning2018modern,bungert2020variational}. 

In recent years, data-driven approaches have emerged, enabling learned regularizers to be incorporated into the variational framework \cite{aharon2006k,chen2014insights,li2020nett,lunz2018adversarial,mukherjee2020learned,xu2012low}. The parameter $\theta$ in these methods could represent anything from the coefficients of a sparsifying synthesis dictionary to the parameters of a neural network. One prominent approach for learning parameters in regularizers is bilevel learning \cite{calatroni2017bilevel,crockett2022bilevel,ehrhardt2023optimal,kunisch2013bilevel, ochs2016techniques}. In this framework, given a set of training data pairs $(x^*_i,u_i)_{i=1,\dots,n}$, the goal is to learn the parameters $\theta$ by solving the bilevel optimization problem:
\rvv{
\begin{equation}\label{bilv-prob}
    \begin{aligned}
        \min_{\theta}\quad &\left\{\mathcal{L}(\theta):=\frac{1}{n}\sum_{i=1}^n\ell(\hat{x}_i(\theta))\right\}\\ 
\text{s.t.} \quad & \hat{x}_i(\theta)=\arg\min_x\Phi_{i}(x;\theta),
    \end{aligned}
\end{equation}
}
\rvv{where $\Phi_{i}(x;\theta)\coloneqq D(Ax,u_i)+R_{\theta}(x)$ denotes the lower-level objective function for each data sample} $i=1,\dots,n$, the reconstruction  $\hat{x}_i(\theta)$ is the solution to the lower-level optimization problem \rvv{with respect to measurement $u_i$}, and $\ell$ is a function that quantifies the quality of the reconstruction, such as the mean squared error  $\|\hat{x}_i(\theta)-x^*_i\|^2$. 

\rvv{We focus on the setting where the learnable parameter $\theta$ is a linear operator $K$ and the lower-level problem in \eqref{bilv-prob} can be formulated as the following general optimization problem:
\begin{equation}\tag{P}\label{piggy_lower}
    \min_{x\in \mathcal{X}} \{\Phi(x;K):=f(Kx) + g(x)\},
\end{equation}
where $f$ and $g$ are proper, convex, lower semicontinuous functions. For example, this corresponds to $R_{\theta}(x)=\psi(Wx)$ in the Fields of Experts model, with the filters $W$ and the activation $\psi$ acting as the learnable operator $K$ and the function $f$ respectively, while the data fidelity acts as $g$. In the setting of training an input-convex neural network, \eqref{piggy_lower} corresponds to a reformulation of the lower-level problem in \eqref{bilv-prob}, which will be presented in later section.}

Bilevel learning is mathematically appealing as it optimizes the parameters based on target data in an explicit measure of reconstruction quality. However, it poses significant computational challenges. One major difficulty is the fact that evaluating the upper-level cost requires solving the lower-level optimization problem exactly. However, exact solutions to the lower-level problem are typically unattainable in practice, as they are often solved numerically up to a certain tolerance. Achieving very small tolerances is computationally prohibitive, especially for large-scale problems. This necessitates the consideration of inexact solutions with controllable tolerances that are as large as possible to reduce computational cost, yet small enough to maintain the accuracy required for effective optimization.

One common class of methods for solving bilevel problems is gradient-based optimization. The key ingredient in these methods is the computation of the gradient of the upper level cost with respect to the parameters $\theta$, which is often referred to as the hypergradient. There are several common approaches for the computation of the hypergradient in the literature. The implicit function theorem (IFT) \cite{bengio2000gradient,ehrhardt2024analyzing,franceschi2018bilevel,grazzi2020iteration,pedregosa2016hyperparameter} approach involves differentiating the optimality condition of the lower-level problem, but requires inverting the Hessian matrix of the lower-level objective via iterative algorithms such as conjugate gradient method, which is computationally expensive. Unrolling methods \cite{bolte2021nonsmooth,bolte2022automatic,kofler2023learning,mehmood2020automatic,ochs2016techniques}, on the other hand compute the hypergradient using automatic differentiation by unrolling the iterations of the optimization algorithm used to solve the lower-level problem. However, this approach can be memory-intensive, especially when a larger number of iterations is used for solving the lower-level problem. The piggyback method \cite{Bogensperger2022,AntoninChambolle2021,Piggyback_classic} computes the hypergradient by solving an adjoint problem, often in conjunction with the lower-level optimization. In this approach, the lower-level solution and the corresponding adjoint problem are solved simultaneously. 

\rv{A class of methods for solving bilevel problems involves reducing them to single-level formulations. One such approach, proposed in \cite{valkonen_single_step, singleloopValkonen}, utilizes a single step of a primal-dual algorithm to approximate the lower-level solution. The adjoint problem, arising from the IFT approach, is then solved using either Conjugate Gradient (CG) with high accuracy or an analytic exact inversion of the lower-level Hessian in \cite{valkonen_single_step}, while \cite{singleloopValkonen} employs a general linear solver. Although these methods do not require a highly accurate lower-level solution at each iteration, they do not provide adaptivity in determining the necessary accuracy for solving the lower-level problem to ensure progress in the upper-level iterations. This results in a high number of upper-level iterations, and their global convergence relies on assumptions regarding initialisation.}

\rv{
Another group of single-level reduction methods, known as fully first-order methods \cite{bome,fullyFirstOrderStochastic,penaltyMethodsNonconvexBO}, replaces the lower-level problem with a so-called value function \cite{dempe_bilevel_2020}. These methods reformulate the bilevel problem as a constrained optimization problem, avoiding the need for IFT-based evaluations and thus eliminating the computational burden of computing or approximating the inverse Hessian of the lower-level problem. Instead, they approximate the lower-level solution using a fixed number of iterations across all upper-level iterations. However, they require careful tuning of step-sizes and penalty multipliers to ensure convergence. Notably, the majority of works employing these approaches focus on stochastic bilevel problems.}

In this paper, we focus on the piggyback method for hypergradient computation, and address several challenges that arise in its practical implementation. First, we propose a framework that allows inexact solutions of the lower-level problem by dynamically adjusting the tolerance during optimization. This helps balancing computational cost and accuracy. Additionally, while piggyback algorithms have been successfully applied, there are no a-posteriori error bounds to control the accuracy of the computed hypergradient. We address this gap by deriving such error bounds, providing a framework for managing inexactness in a principled manner.

Another challenge in applying gradient-based methods to bilevel problems is the choice of step-size for the upper-level optimization. Convergence analysis for \rv{nested gradient-based} bilevel methods typically assumes highly accurate hypergradients and the use of a fixed, small step-size \cite{pedregosa2016hyperparameter,ghadimi2018approximation, ji2021bilevel}. In the context of piggyback-style computed hypergradients, existing analyses also require highly accurate hypergradients and a sufficiently small, fixed step-size \cite{Bogensperger2022}. However, in practice, determining an appropriate step-size when working with approximate hypergradients is non-trivial. Assuming access to inexact lower-level solutions with controllable error and inexact hypergradients with a-posteriori error bounds, a condition was developed in \cite{salehi2025adaptivelyinexactfirstordermethod} to adaptively select step-sizes and the required accuracy, ensuring convergence in a robust and cost-efficient manner. To connect and apply the method in \cite{salehi2025adaptivelyinexactfirstordermethod} to primal-dual style differentiation, we propose an a-posteriori error bound analysis for the lower-level and adjoint solutions, as well as a-posteriori error bounds for the inexact hypergradient.
We demonstrate the effectiveness of our approach through some bilevel learning tasks, including learning the discretization of the total variation (TV) and training input-convex neural networks (ICNNs) regularizers \cite{amos2017input,mukherjee2020learned}. The proposed method provides a framework that balances computational efficiency with rigorous error control, making it well-suited for large-scale bilevel learning applications.
\\ \\
\rv[final]{To summarize, our contributions are as follows:
\begin{itemize}
    \item We provide a detailed theoretical analysis establishing a-posteriori controllable error bounds for primal-dual-style bilevel differentiation, extending the framework in~\cite{salehi2025adaptivelyinexactfirstordermethod}.
    \item We prove a convergence result for primal-dual-style bilevel methods, which, to the best of our knowledge, had not been established before.
    \item We demonstrate the suitability of our method to learn improved discretizations of the Total Variation (TV) and to train input convex neural networks (ICNNs) as regularizers via bilevel learning. In the latter case, we show improved performance over previous methods for training such data-adaptive regularizers.
\end{itemize}}

\subsection{Notation}
We denote the Euclidean norm of vectors and the $2$-norm of matrices with $\| \cdot \|$. The conjugate of a function $f$ is given by $f^*(y) = \sup_{x}  \langle x, y \rangle - f(x)
$. The constant $\mu_f$ represents the strong convexity constant of a strongly convex function $f$. The gradient and Hessian of a function $f$ are denoted by $\nabla f$ and $\nabla^2 f$, respectively. Additionally, for a function $f$, the Lipschitz and H\"older continuity constants of the gradient and Hessian are denoted by $L_{\nabla f}$ and $L_{\nabla^2 f}$, respectively.

\section{Computation of Hypergradient}\label{sec:Piggy}
\subsection{Problem Setting}
The corresponding dual problem of \eqref{piggy_lower} is given by:
\begin{equation}\tag{D}\label{piggy_lower_dual}
    \sup_{y\in \mathcal{Y}} -f^*(y) - g^*(-K^*y),
\end{equation}
where $f^*,g^*$ are the convex conjugates of $f,g$ respectively. A standard approach in convex optimization literature \cite{chambolle2016introduction,chambolle2011first} is to formulate this as a saddle-point problem:
\begin{equation}\tag{S}\label{saddle}
    \min_{x\in \mathcal{X}}\max_{y\in \mathcal{Y}} \{h(x,y \rv{;} K) \coloneqq \langle Kx,y\rangle + g(x) - f^*(y)\}.
\end{equation}
Assuming a solution $(\hat{x}(K),\hat{y}(K))$ of \eqref{saddle} exists, then $\hat{x}(K)$ is also a solution to \eqref{piggy_lower}. The first-order optimality condition can be stated as:
\begin{equation}\label{Arrow-Hurwicz}
    \begin{cases}
    0&\in K\hat{x}(K) - \partial f^*(\hat{y}(K)) \\
    0&\in K^*\hat{y}(K) + \partial g(\hat{x}(K)).
\end{cases}
\end{equation}

In this paper, we aim to learn a linear operator $K$, or specific components of it, in a supervised fashion, via solving the following bilevel problem:
\begin{subequations}\label{piggy_bilevel}
\begin{align}
        \min_{K} \{\mathcal{L}(K) &\coloneqq \ell (\hat{x}(K), \hat{y}(K))\}\label{upperpiggy}\\
         {(\hat{x}(K), \hat{y}(K))} &\coloneqq \arg\min_{x\in \mathcal{X}}\max_{y\in \mathcal{Y}} h(x,y\rv{;}K) .\label{lowerpiggy}
    \end{align}
\end{subequations}

\rv{To define the upper-level loss with respect to the lower-level solution, we assume that $(\hat{x}(K),\hat{y}(K))$ given $K$ is unique. This can be guaranteed if $f^*,g$ are strongly convex. \begin{assumption}\label{ass_strong-cvx}
    $f^*$ and $g$ are $\mu_{f^*}$- and $\mu_g$-strongly convex, respectively.
\end{assumption}
\begin{remark}\label{remark1}
    In particular, this implies that $f$, $g^*$ have Lipschitz gradient, i.e. $f$, $g^*$ are $C^{1,1}$ \cite{zhou2018fenchel}.
 \end{remark}
}
Here, the upper-level loss function with respect to the lower-level solution is represented by $\ell$. Throughout the paper, we make the following assumption on $\ell$.

\begin{assumption}\label{ass_pigy_1}
    The upper-level loss $\ell: \mathcal{X}\times \mathcal{Y} \to \mathbb{R}_+$ is separable, that is, $\ell(x,y)=\ell_1(x)+\ell_2(y)$. Moreover, $\ell_1,\ell_2$ are assumed to be $L_{1},L_{2}$-smooth respectively, that is, $\nabla\ell_1,\nabla\ell_2$ are $L_{1},L_{2}$ Lipschitz.
\end{assumption}

\begin{remark}
     We particularly consider $\ell(x,y)=\ell(x)$ for the bilevel learning examples. However, the a-posteriori error bound can be extended to more general $\ell$, where we assume $\nabla\ell$ is $L$ Lipschitz.
 \end{remark}
Since the upper-level problem (\ref{upperpiggy}) can be large-scale, we are interested in gradient-based bilevel methods \cite{crockett2022bilevel,pedregosa2016hyperparameter} which employ $\nabla \mathcal{L}(K)$ to solve (\ref{upperpiggy}). Under \Cref{ass_strong-cvx} and \ref{ass_pigy_1}, $(\hat{x}(K),\hat{y}(K))$ given $K$ is unique, and the hypergradient can be calculated as follows \cite{Bogensperger2022}:
\begin{equation}\label{hypergrad_piggy}
    \nabla \mathcal{L}(K) = \hat{y}(K) \otimes \hat{X}(K) + \hat{Y}(K) \otimes \hat{x}(K),
\end{equation}
where the adjoint variables $(\hat{X}(K) , \hat{Y}(K) ) \in \mathcal{X} \times \mathcal{Y}$ solve the quadratic adjoint saddle-point problem below \small{
\begin{equation}\tag{AS}\label{saddle_adjoint}
\begin{aligned}
    \min_{X\in \mathcal{X}} \max_{Y \in \mathcal{Y}} \ &\langle KX, Y \rangle + \frac{1}{2}\langle \nabla^2g(\hat{x}(K))X , X \rangle \\ 
    &- \frac{1}{2}\langle\nabla^2f^*(\hat{y}(K))Y, Y \rangle\\
    &+\langle \nabla\ell_1(\hat{x}(K)),X\rangle+\langle \nabla\ell_2(\hat{y}(K)),Y\rangle.
\end{aligned}
\end{equation}
}\normalsize

\rv{Given a convex function $g$, Alexandrov's theorem \cite{howard1998alexandrov} states that the set-valued subgradient map $\partial g$ is differentiable almost everywhere. Moreover, at a point of differentiability $x$, $\partial g(x)$ is single valued and hence $\nabla g(x)$ exists. We then say that $\nabla^2g(x)$ exists and is equal to $\nabla(\nabla g)(x)$. In particular, $\nabla^2g(x)$ is non-singular \cite{Bogensperger2022} if $g$ is also strongly convex. In \eqref{saddle_adjoint} and what follows, both $\nabla^2g$ and $\nabla^2f^*$ will be interpreted in this sense.}

The following auxiliary lemma will be used in our analysis to provide an a-posteriori control over the optimality gap for strongly convex functions.
\begin{lemma}\label{lm:aposteriori_sc}
     (\cite{ehrhardt2021inexact}) Let $\Phi : \mathcal{X} \rightarrow \mathbb{R}$ be $\mu$-strongly convex and differentiable\footnote{\rv{The above lemma still holds when $\Phi$ is just $\mu$-strongly convex. The gradient $\nabla\Phi$ can be replaced by any subgradient $g_x\in\partial\Phi(x)$ \cite[Theorem 5.24(iii)]{beck2017first}.}}. Denoting the minimizer of $\Phi$ by $x^* \in \mathcal{X}$, for all $x \in \mathcal{X}$ we have
     \begin{equation*}
           \lVert x^* - x \rVert \leq \frac{1}{\mu} \lVert \nabla \Phi(x) \rVert.
     \end{equation*}
\end{lemma}


\begin{lemma}\label{lm:aposteriori_sc_saddle}
    Let \Cref{ass_strong-cvx} hold. For a given operator $K$, and considering the unique saddle-point $(\hat{x}, \hat{y})$ of \eqref{saddle}, for all $(x, y) \in \mathcal{X}\times \mathcal{Y}$ we have\footnote{\rv{By slight abuse of notation, we still denote any subgradients with $\nabla g(x)$ and $\nabla f^*(y)$.}}

\begin{subequations}\label{bound_x_y}
    \begin{align}
         \|x - \hat{x}\| &\leq \frac{\|\nabla g(x) + K^*\nabla f(Kx)\|}{\mu_g}\label{bound_1}, \\
         \|y - \hat{y}\| &\leq \frac{\|\nabla f^*(y)-K\nabla g^*(-K^*y)\|}{\mu_{f^*}}\label{bound_2}.
    \end{align}
\end{subequations}
\end{lemma}
\begin{proof}
Since $\hat{x}$ is the minimizer of \eqref{piggy_lower}, and from \Cref{ass_strong-cvx} the primal problem's objective $f(K\cdot)+ g(\cdot)$ is $\mu_g$-strongly convex, utilizing \Cref{lm:aposteriori_sc} one can derive \eqref{bound_1}. On the other hand, we note that the dual problem \eqref{piggy_lower_dual} is equivalent to :
\begin{equation}\label{piggy_lower_dual_min}
    \min_{y\in \mathcal{Y}}f^*(y)+g^*(-K^*y),
\end{equation}
where the objective $f^*(\cdot)+g^*(-K^*\cdot)$ is $\mu_{f^*}$-strongly convex. Since $\hat{y}$ is the minimizer of \eqref{piggy_lower_dual_min}, utilizing \Cref{lm:aposteriori_sc} again yields the bound \eqref{bound_2} as required.
\end{proof} 

\subsection{Piggyback Algorithm}
To compute the hypergradient $\nabla \mathcal{L}(K)$, two variations of the piggyback algorithm were introduced in \cite{AntoninChambolle2021}. By looking at \eqref{hypergrad_piggy}, computing the hypergradient involves solving two saddle-point problems: \eqref{saddle} and \eqref{saddle_adjoint}, respectively. A well-suited algorithm to solve each of these problems is the Primal-Dual hybrid gradient (PDHG\footnote{\rv{The authors in \cite{PDHG_category} referred to this algorithm as the modified primal-dual hybrid gradient (PDHGM).}}) \cite[Algorithm 3]{chambolle2011first}. Since the adjoint problem \eqref{saddle_adjoint} depends on the solution of \eqref{saddle}, two different approaches can be taken. One strategy is to solve \eqref{saddle} and \eqref{saddle_adjoint} alternatively, where each iterations includes a step of PDHG applied to \eqref{saddle}, followed by plugging the result in \eqref{saddle_adjoint} and performing one step of inexact PDHG applied to the adjoint problem, as detailed in \cite{rasch2020inexact}. 

Alternatively, PDHG can be used to solve the saddle-point problem \eqref{saddle} first. Since the exact solution $(\hat{x}(K), \hat{y}(K))$ of \eqref{saddle} is not attainable in finite time, an approximate saddle-point $(\tilde{x}(K), \tilde{y}(K))$ is computed instead. With this approximation, we can then solve the saddle-point problem below:
\begin{equation}\tag{ASI}\label{saddle_adjoint_Inexact}
\begin{aligned}
    \min_{X\in \mathcal{X}} \max_{Y \in \mathcal{Y}} \ &\langle KX, Y \rangle + \frac{1}{2}\langle \nabla^2g(\tilde{x}(K))X , X \rangle \\ 
    &- \frac{1}{2}\langle \nabla^2f^*(\tilde{y}(K))Y, Y \rangle\\
    &+\langle \nabla\ell_1(\tilde{x}(K)),X\rangle+\langle \nabla\ell_2(\tilde{y}(K)),Y\rangle,
\end{aligned}
\end{equation} corresponding to the adjoint variables, given an inexact saddle-point of \eqref{saddle}, using PDHG to find the saddle-point $(\bar{X}(K), \bar{Y}(K))$ of \eqref{saddle_adjoint_Inexact}. Finally, the approximate hypergradient can be calculated as 
\begin{equation}
    z(K) \coloneqq \tilde{y}(K)\otimes \bar{X}(K) + \bar{Y}(K) \otimes \tilde{x}(K).
\end{equation} 

\begin{algorithm*}[h]
\caption{Piggyback Algorithm (suggested in \cite{Bogensperger2022}). Hyperparameter: $\nu \leq \frac{2\sqrt{\mu_g \mu_{f^*}}}{\|K\|}$. }
\label{alg:piggyback_2}
\begin{algorithmic}[1]
\State Input: $(x_0, y_0), (X_0, Y_0) \in \mathcal{X} \times \mathcal{Y}$.
\State{Set \rv{step-sizes} $\tau = \frac{\nu}{2\mu_g}$, $\sigma = \frac{\nu}{2\mu_{f^*}}$, and \rv{extrapolation factor }$\theta \in [\frac{1}{1+\nu},1]$.}
\State{\rv{Apply \cite[Algorithm 3]{chambolle2011first} to \eqref{saddle} with $\tau,\sigma,\theta$ chosen above to find approximate solution $(\tilde{x}(K), \tilde{y}(K))$.}}\label{inner_step}
\State{\rv{Apply \cite[Algorithm 3]{chambolle2011first} to \eqref{saddle_adjoint_Inexact} with $\tau,\sigma,\theta$ chosen above to find approximate solution $(\tilde{X}(K), \tilde{Y}(K))$.}}\label{outer_step}
\State{Update the hypergradient approximation $z(K) = \tilde{y}(K)\otimes \tilde{X}(K) + \tilde{Y}(K) \otimes \tilde{x}(K)$.}
\end{algorithmic}
\end{algorithm*}


\rv{To ensure a well-defined hypergradient that can be approximated with theoretical guarantees, similar to \cite{Bogensperger2022}, we impose additional regularity assumption on $g^*$ and $f$.}
\rv{
\begin{definition}
A $C^2$ function $f$ is said to be $C^{2,\alpha}$ if for some $L_{\nabla^2f}>0$, the following holds for all $x_1,x_2$:
\begin{equation}
    \|\nabla^2f(x_1)-\nabla^2f(x_2)\|\leq L_{\nabla^2f}\|x_1-x_2\|^\alpha
\end{equation}
\end{definition}
}
\rv{
\begin{assumption}\label{ass_pigy_2}
$f$ and $g^*$ are locally $C^{2,1}$.\footnote{For simplicity, we consider locally $C^{2,1}$ functions, but the analysis can be generalized to locally $C^{2,\alpha}$ functions for $\alpha\in(0,1]$, similar to \cite{Bogensperger2022}.}
\end{assumption}
}
It was shown in \cite[Theorem 2.2]{Bogensperger2022} that under \Cref{ass_strong-cvx}, \ref{ass_pigy_1}, and \ref{ass_pigy_2}, \rvv{and following the parameter selection condition on the step-sizes $\tau,\sigma$, and extrapolation factor $\theta$ of PDHG \cite[Algorithm 3]{chambolle2011first}}, this alternative piggyback algorithm converges linearly to $(\hat{x}(K), \hat{y}(K), \hat{X}(K), \hat{Y}(K))$, and thus to the hypergradient $\nabla \mathcal{L}(K)$.

This alternative approach is outlined in \Cref{alg:piggyback_2}. The error bound for the resulting approximate hypergradient is given in the following theorem.
\begin{theorem}\label{th_piggy}(\cite[Theorem 2.3]{Bogensperger2022})
    Under \Cref{ass_strong-cvx}, \ref{ass_pigy_1}, and \ref{ass_pigy_2}, let  $\nu \leq \frac{2\sqrt{\mu_g \mu_{f^*}}}{\|K\|}$, $\tau = \frac{\nu}{2\mu_g}$, $\sigma = \frac{\nu}{2\mu_{f^*}}$, $\theta \in [\frac{1}{1+\nu},1]$ chosen as in \cite[Algorithm 3]{chambolle2011first}. The iterates of \Cref{alg:piggyback_2} satisfy
\begin{multline*}
\|z(K) - \nabla \mathcal{L}(K)\| \leq C \Big( \big(\|\tilde{x}(K) - \hat{x}(K)\| \\ + \|\tilde{y}(K) - \hat{y}(K)\|\big)^\alpha 
+ \omega^{\frac{k}{2}} \Big),
\end{multline*}
for some constant $C>0$, rate $\omega \leq \theta$ and $\alpha \in (0,1]$.
\end{theorem}
Although \Cref{th_piggy} shows the iterates of \Cref{alg:piggyback_2} converge to the hypergradient $\nabla\mathcal{L}(K)$, the practical challenges still persist. Running too many iterations in steps \ref{inner_step} and \ref{outer_step} of \Cref{alg:piggyback_2} can lead to an impractically slow performance. Conversely, using too few iterations may cause a loss of convergence in the gradient descent method for the upper-level problem, as it would rely on the approximated hypergradient $\tilde{y}(K) \otimes \tilde{X}(K)+\tilde{Y}(K) \otimes \tilde{x}(K)$ for solving \eqref{upperpiggy}. To tackle this issue, we establish an a-posteriori analysis to quantify the error in the computed hypergradient, allowing us to derive a reliable stopping criterion for steps \ref{inner_step} and \ref{outer_step} of \Cref{alg:piggyback_2}.

\section{Adaptive Bilevel Learning}
\subsection{Inexact Piggyback Algorithm and A-Posteriori Error Analysis}
Since our focus is on the a-posteriori error bound of the hypergradient evaluated at a given operator $K$, we simplify the notation in this section by dropping the dependence on $K$ for brevity. Specifically, we denote $\hat{x}(K), \hat{y}(K), \tilde{x}(K), \tilde{y}(K)$ by $\hat{x}, \hat{y}, \tilde{x}, \tilde{y}$. Similarly, we denote $\hat{X}(K), \hat{Y}(K), \bar{X}(K), \bar{Y}(K)$ by $\hat{X}, \hat{Y}, \bar{X}, \bar{Y}$.

\begin{lemma}\label{lm:aposteriori_sc_saddle_adjoint}
    Let \Cref{ass_strong-cvx} and \ref{ass_pigy_1} hold. For a given operator $K$, and consider the unique saddle-point $(\bar{X}, \bar{Y})$ of \eqref{saddle_adjoint_Inexact}, we have that, for all $(X,Y) \in \mathcal{X} \times \mathcal{Y}$,
    \begin{subequations}\label{bound_X_Y}
    \begin{align}
         &\|X - \bar{X}\|\leq\delta_1(X,\tilde{x},\tilde{y})\notag \\ 
         &:=\frac{\|\bar{B}_1X + \nabla\ell_1(\tilde{x})+ K^*\nabla^2f^*(\tilde{y})^{-1}\nabla\ell_2(\tilde{y})\|}{\mu_g}\label{bound_1_ad}, \\
         &\|Y - \bar{Y}\|\leq\delta_2(Y,\tilde{x}.\tilde{y})\notag \\ &:= \frac{\|\bar{B}_2Y + K\nabla^2g(\tilde{x})^{-1}\nabla\ell_1(\tilde{x})-\nabla\ell_2 (\tilde{y})\|}{\mu_{f^*}}\label{bound_2_ad},
    \end{align}
\end{subequations}
where $\bar{B}_1=\nabla^2g(\tilde{x})+K^*\nabla^2f^*(\tilde{y})^{-1}K$, and\\ $\bar{B}_2=K\nabla^2g(\tilde{x})^{-1}K^*+\nabla^2f^*(\tilde{y})$.
\end{lemma}
\begin{proof}
    The primal and dual problem correspond to the inexact adjoint saddle-point problem \eqref{saddle_adjoint_Inexact} can be written as: 
    \begin{multline}\label{primal_ASI}
         \min_{X\in \mathcal{X}} \frac{1}{2}\langle \bar{B}_1X , X \rangle +\langle\nabla\ell_1 (\tilde{x}), X\rangle\\+\langle K^*\nabla^2f^*(\tilde{y})^{-1}\nabla\ell_2(\tilde{y}),X\rangle\\
         +\frac{1}{2}\langle\nabla\ell_2(\tilde{y}),\nabla^2f^*(\tilde{y})^{-1}\nabla\ell_2(\tilde{y})\rangle,
     \end{multline}
    \begin{multline}\label{dual_ASI}
    \max_{Y\in\mathcal{Y}} -\frac{1}{2}\langle \bar{B}_2Y,Y\rangle-\langle K\nabla^2g(\tilde{x})^{-1}\nabla\ell_1(\tilde{x}),Y\rangle\\
+\langle\nabla\ell_2(\tilde{y}),Y\rangle-\frac{1}{2}\langle\nabla\ell_1(\tilde{x}),\nabla^2g(\tilde{x})^{-1}\nabla\ell_1(\tilde{x})\rangle.        
    \end{multline}
Using \Cref{ass_strong-cvx}, \ref{ass_pigy_1}, and \Cref{lm:aposteriori_sc} we derive \eqref{bound_1_ad} and \eqref{bound_2_ad} as required.
\end{proof}
\begin{lemma}\label{locally_lip_hess}
    Let $g$ be $\mu_g$-strongly convex and $L_{\nabla g}$-smooth, and let $g^*$ be $C^{2,1}$. Then $\nabla^2g(x)$ is Lipschitz continuous with constant $L_{\nabla^2g^*}(L_{\nabla g})^{3}$. 
\end{lemma}
\begin{proof}
For $x_1, x_2$, let $y_1=\nabla g(x_1), y_2=\nabla g(x_2)$, we have:
\begin{align*}
&\|\nabla^2g(x_1)-\nabla^2g(x_2)\|\\
&=\|\nabla^2g(x_1)(\nabla^2g(x_2)^{-1}-\nabla^2g(x_1)^{-1})\nabla^2g(x_2)\| \\ 
&\leq\|\nabla^2g(x_1)\| \cdot \|\nabla^2g^*(y_2)-\nabla^2g^*(y_1)\|\cdot\|\nabla^2g(x_2)\|\\
&\leq(L_{\nabla g})^2\|\nabla^2g^*(y_1)-\nabla^2g^*(y_2)\|.
\end{align*}
\rvv{For the first inequality, we utilized the fact that $\nabla^2g(x_i)^{-1}=\nabla^2g^*(y_i)$ for $i=1,2$}, and the last inequality is due to \cite[Theorem 1]{zhou2018fenchel} and the $L_{\nabla g}$ smoothness of $g$.
Since $g^*$ is $C^{2,1}$ and $g$ is $L_{\nabla g}$ smooth, we have:
\begin{multline*}
\|\nabla^2g(x_1)-\nabla^2g(x_2)\|\\
\shoveleft\leq(L_{\nabla g})^2\|\nabla^2g^*(y_1)-\nabla^2g^*(y_2)\|\\
\shoveleft\leq L_{\nabla^2g^*}(L_{\nabla g})^2\|\nabla g(x_1)-\nabla g(x_2)\|\\
\shoveleft\leq L_{\nabla^2g^*}(L_{\nabla g})^3\|x_1-x_2\|.\\
\end{multline*}
\end{proof}
\begin{theorem}\label{apos_piggy_theorem}
   Suppose \Cref{ass_strong-cvx}, \ref{ass_pigy_1}, and \ref{ass_pigy_2} hold. Moreover, let $g$ and $f^*$ be $L_{\nabla g}$ and $L_{\nabla f^*}$ smooth, respectively. For $\epsilon^x,\epsilon^y,\delta^X,\delta^Y>0$, if $\|\tilde{x} - \hat{x}\|\leq \epsilon^x$, $\|\tilde{y} - \hat{y}\|\leq \epsilon^y$. Let $(\tilde{X},\tilde{Y})\in\mathcal{X}\times\mathcal{Y}$ such that  
   \begin{subequations}
       \begin{align}
           \|\tilde{X}-\bar{X}\|&\leq \delta_1(\tilde{X},\tilde{x},\tilde{y})\leq\delta^X\\
           \|\tilde{Y}-\bar{Y}\|&\leq \delta_2(\tilde{Y},\tilde{x},\tilde{y})\leq\delta^Y,
       \end{align}
   \end{subequations}
   then $(\tilde{X},\tilde{Y})$ satisfies:
\begin{equation}\label{bound_adjoint_piggy_1}
     \|\tilde{X} - \hat{X}\| \leq  C^X_1\epsilon^x+C^X_2\epsilon^y+\delta^X,
\end{equation}
\begin{equation}\label{bound_adjoint_piggy_2}
     \|\tilde{Y} - \hat{Y}\| \leq  C^Y_1\epsilon^x+C^Y_2\epsilon^y+\delta^Y.
\end{equation}
For 
\begin{equation}
    z :=\tilde{y}\otimes \tilde{X} + \tilde{Y} \otimes \tilde{x},
\end{equation}
we have:
\begin{equation}\label{error_bound_piggy}
\begin{aligned}
&\|z - \nabla\mathcal{L}(K)\|\leq
(C^Y_1\|\tilde{x}\|+C^X_1\|\tilde{y}\|+\|\tilde{Y}\|)\epsilon^x\\
&+(C^Y_2\|\tilde{x}\|+C^X_2\|\tilde{y}\|+\|\tilde{X}\|)\epsilon^y+\|\tilde{y}\|\delta^X+\|\tilde{x}\|\delta^Y\\
&+C^Y_1(\epsilon^x)^2+C^X_2(\epsilon^y)^2+\delta^X\epsilon^y+\delta^Y\epsilon^x
\end{aligned}
\end{equation}
where 
\begin{subequations}
\begin{align}
&C^X_1:=\frac{L_{\nabla^2g^*}(L_{\nabla g})^{3}\|\tilde{X}\|+L_{1}}{\mu_g}\label{CX1}\\
&C^X_2:=\notag\\
&\frac{L_{\nabla^2f}L_{\nabla f^*}\|K\|(\|K\|\|\tilde{X}\|+\|\nabla\ell_2(\tilde{y})\|)}{\mu_g}+\frac{L_{2}\|K\|}{\mu_g\mu_{f^*}}\label{CX2}\\
&C^Y_1:=\notag\\
&\frac{L_{\nabla^2g^*}L_{\nabla g}\|K\|(\|K\|\|\tilde{Y}\|+\|\nabla\ell_1(\tilde{x})\|)}{\mu_{f^*}}+\frac{L_{1}\|K\|}{\mu_g\mu_{f^*}}\label{CY1}\\
&C^Y_2:=\frac{L_{\nabla^2f}(L_{\nabla f^*})^{3}\|\tilde{Y}\|+L_{2}}{\mu_{f^*}}
\label{CY2}
\end{align}
\end{subequations}
\end{theorem}

\begin{proof}
Taking similar steps as \rvv{\Cref{lm:aposteriori_sc_saddle_adjoint}} on \eqref{saddle_adjoint} and defining
$$\hat{B}_1=\nabla^2g(\hat{x})+K^*\nabla^2f^*(\hat{y})^{-1}K,$$ $$\hat{B}_2=K\nabla^2g(\hat{x})^{-1}K^*+\nabla^2f^*(\hat{y}),$$
we derive:
\begin{equation}\label{aposXk_exact}
     \begin{aligned}
     &\|\tilde{X} - \hat{X}\| \\ 
         &\leq \frac{\|\hat{B}_1\tilde{X} + \nabla\ell_1(\hat{x}) + K^*\nabla^2f^*(\hat{y})^{-1}\nabla\ell_2(\hat{y})\|}{\mu_g},    
     \end{aligned}   
\end{equation}
and 
\begin{equation}\label{aposYk_exact}
     \begin{aligned}
         &\|\tilde{Y} - \hat{Y}\| \\ 
         &\leq \frac{\|\hat{B}_2\tilde{Y} + K\nabla^2g(\hat{x})^{-1}\nabla\ell_1(\hat{x})-\nabla\ell_2 (\hat{y})\|}{\mu_{f^*}},
     \end{aligned}
\end{equation}
where $(\hat{X}, \hat{Y})$ is the saddle-point of \eqref{saddle_adjoint}. \rvv{Now, from the RHS of \eqref{aposXk_exact}, by adding and subtracting $\tilde{B}\tilde{X}+\nabla_1(\tilde{x})$, we have:}
\begin{equation}\label{aposXk_ieq1}
\begin{aligned}
&\|\hat{B}_1\tilde{X} + \nabla\ell_1 (\hat{x})+K^*\nabla^2f^*(\hat{y})^{-1}\nabla\ell_2(\hat{y})\|\\
&=\|\bar{B}_1\tilde{X}+\nabla\ell_1 (\tilde{x})+K^*\nabla^2f^*(\tilde{y})^{-1}\nabla\ell_2(\tilde{y})\\
&+(\hat{B}_1-\bar{B}_1)\tilde{X}+\nabla\ell_1 (\hat{x})-\nabla\ell_1(\tilde{x})\\
&+K^*\nabla^2f^*(\hat{y})^{-1}\nabla\ell_2(\hat{y})-K^*\nabla^2f^*(\tilde{y})^{-1}\nabla\ell_2(\tilde{y})\|\\ &\leq\|\bar{B}_1\tilde{X}+\nabla\ell_1(\tilde{x})+K^*\nabla^2f^*(\tilde{y})^{-1}\nabla\ell_2(\tilde{y})\|\\
&+\|\hat{B}_1-\bar{B}_1\|\|\tilde{X}\|+L_{1}\epsilon^x\\
&+\|K^*\nabla^2f^*(\hat{y})^{-1}\nabla\ell_2(\hat{y})-K^*\nabla^2f^*(\tilde{y})^{-1}\nabla\ell_2(\tilde{y})\|
\end{aligned}
\end{equation}
where the inequality is due to the $L_{1}$-smoothness of $\ell_1$ \rvv{and the triangle inequality.} Note that we have:
\begin{equation}\label{aposXk_ieq2}
    \begin{aligned}
        &\|\hat{B}_1-\bar{B}_1\|\\
        &=\|\nabla^2g(\hat{x})-\nabla^2g(\tilde{x})\\
        &+K^*(\nabla^2f(\nabla f^*(\hat{y}))-\nabla^2f(\nabla f^*(\tilde{y})))K\|\\
        &\leq L_{\nabla^2g^*}(L_{\nabla g})^{3}\|\hat{x}-\tilde{x}\|\\
        &+L_{\nabla^2f}\|K\|^2\|\nabla f^*(\hat{y})-\nabla f^*(\tilde{y})\|\\
        &\leq L_{\nabla^2g^*}(L_{\nabla g})^{3}\epsilon^{x}+L_{\nabla^2f}L_{\nabla f^*}\|K\|^2\|\hat{y}-\tilde{y}\|\\
        &=L_{\nabla^2g^*}(L_{\nabla g})^{3}\epsilon^{x}+L_{\nabla^2f}L_{\nabla f^*}\|K\|^2\epsilon^y,
    \end{aligned}
\end{equation}
\rvv{where the first inequality is due to \Cref{locally_lip_hess} and the Lipschitz condition on $\nabla^2f$, and the last inequality is due to \Cref{ass_strong-cvx} and \Cref{remark1}.}
We note that:
\begin{equation}\label{aposXk_ieq3}
\begin{aligned}
&\|K^*\nabla^2f^*(\hat{y})^{-1}\nabla\ell_2(\hat{y})-K^*\nabla^2f^*(\tilde{y})^{-1}\nabla\ell_2(\tilde{y})\|\\
&=\|K^*\nabla^2f^*(\hat{y})^{-1}\nabla\ell_2(\hat{y})-K^*\nabla^2f^*(\hat{y})^{-1}\nabla\ell_2(\tilde{y})\\
&+K^*\nabla^2f^*(\hat{y})^{-1}\nabla\ell_2(\tilde{y})-K^*\nabla^2f^*(\tilde{y})^{-1}\nabla\ell_2(\tilde{y})\|\\
&\leq\|K\|\|\nabla^2f^*(\hat{y})^{-1}\|\|\nabla\ell_2(\hat{y})-\nabla\ell_2(\tilde{y})\|\\
&+\|K\|\|\nabla^2f^*(\hat{y})^{-1}-\nabla^2f^*(\tilde{y})^{-1}\|\|\nabla\ell_2(\tilde{y})\|
\end{aligned}
\end{equation}
where the last inequality again utilizes \cite[Theorem 1]{zhou2018fenchel}. Moreover, \rvv{since $\nabla^2f^*(y)^{-1}=\nabla^2f(\nabla f^*(y))$} for all $y$, we have:
\begin{equation}\label{aposXk_ieq4}
\begin{aligned}
&\|\nabla^2f^*(\hat{y})^{-1}-\nabla^2f^*(\tilde{y})^{-1}\|\\
&=\|\nabla^2f(\nabla f^*(\hat{y}))-\nabla^2f(\nabla f^*(\tilde{y}))\|\\
&\leq L_{\nabla^2f}\|\nabla f^*(\hat{y})-\nabla f^*(\tilde{y})\|\\
&\leq L_{\nabla^2f}L_{\nabla f^*}\|\hat{y}-\tilde{y}\|.
\end{aligned}
\end{equation}
Combining \eqref{aposXk_ieq1},\eqref{aposXk_ieq2},\eqref{aposXk_ieq3}, \eqref{aposXk_ieq4}, this implies
\begin{equation}\label{aposX0_exact_final}     \|\tilde{X} - \hat{X}\| \leq  C^X_1\epsilon^x+C^X_2\epsilon^y+\delta^X\end{equation}
where $C^X_1,C^X_2$ are as defined in \eqref{CX1},\eqref{CX2}.
Similarly, taking the same steps on the RHS of \eqref{aposYk_exact} yields
\begin{equation}\label{aposY0_exact_final}
    \|\tilde{Y} - \hat{Y}\| \leq  C^Y_1\epsilon^x+C^Y_2\epsilon^y+\delta^Y,
\end{equation}
with $C^Y_1,C^Y_2$ as \eqref{CY1},\eqref{CY2}.
Now, for the hypergradient $z$, we know
\begin{multline*}
    \|z - \nabla\mathcal{L}(K)\| \leq \|\tilde{y}\|\|\tilde{X} - \hat{X}\| + \|\tilde{y}-\hat{y}\|\|\hat{X}\| \\
    + \|\tilde{x}\|\|\tilde{Y} - \hat{Y}\| + \|\tilde{x}-\hat{x}\|\|\hat{Y}\|.
\end{multline*}
Adding and subtracting $\tilde{X}$ to $\hat{X}$ in the second term and $\tilde{Y}$ to $\hat{Y}$ in fourth term of the RHS above, applying triangle inequality, employing \eqref{aposX0_exact_final} and \eqref{aposY0_exact_final}, we derive the bound \eqref{error_bound_piggy} as required.
\end{proof}

\rvv{
\begin{remark}
For a general $\ell$, $\nabla\ell_1$ and $\nabla\ell_2$ will be replaced by $\nabla_x\ell$ and $\nabla_y\ell$ in \Cref{lm:aposteriori_sc_saddle_adjoint} and \Cref{apos_piggy_theorem} respectively. In the proof of \Cref{apos_piggy_theorem}, error bounds on the gradients of upper-level loss such as $\|\nabla_1(\hat{x})-\nabla_1(\tilde{x})\|$ can be bounded by $L(\epsilon^x+\epsilon^y)$ instead.
\end{remark}}

\rv{Note that \Cref{apos_piggy_theorem}, and in particular \eqref{error_bound_piggy}, provides an a-posteriori and fully computable error bound for the inexact hypergradient calculated through primal-dual style differentiation, which can be used in adaptive bilevel methods such as \cite{salehi2025adaptivelyinexactfirstordermethod}.} Now, utilizing \Cref{lm:aposteriori_sc} and \Cref{apos_piggy_theorem}, we propose an inexact variation of Piggyback algorithm with controllable error on the lower-level iterates as well as the hypergradient, which is outlined in \Cref{alg:piggyback_inexact}.

\begin{algorithm*}[h]
\caption{Inexact Piggyback Algorithm. Hyperparameter: $\nu \leq \frac{2\sqrt{\mu_g \mu_{f^*}}}{\|K\|}$. }
\label{alg:piggyback_inexact}
\begin{algorithmic}[1]
\State Input: $(x_0, y_0), (X_0, Y_0) \in \mathcal{X} \times \mathcal{Y}$, accuracies $\epsilon^x, \epsilon^y, \delta^X, \delta^Y\geq 0$
\Function{InexactPiggyback}{$K,\epsilon^x, \epsilon^y, \delta^X, \delta^Y$}
\State{Set $\tau = \frac{\nu}{2\mu_g}$, $\sigma = \frac{\nu}{2\mu_{f^*}}$, and $\theta \in [\frac{1}{1+\nu},1]$.}
\State{Solve \eqref{saddle} to find $(\tilde{x}, \tilde{y})$, where $\|\tilde{x} - \hat{x}\| \leq \epsilon^x$ and $\|\tilde{y} - \hat{y}\| \leq \epsilon^y$, utilizing \eqref{bound_1} and \eqref{bound_2}}
\State{Solve \eqref{saddle_adjoint_Inexact} 
to find $(\tilde{X}, \tilde{Y})$, where $\|\tilde{X} - \bar{X}\| \leq \delta^X$ and $\|\tilde{Y} - \bar{Y}\| \leq \delta^Y$, utilizing \eqref{bound_adjoint_piggy_1} and \eqref{bound_adjoint_piggy_2}.}
\State\Return $z = \tilde{y} \otimes \tilde{X} + \tilde{Y} \otimes \tilde{x}$
\EndFunction
\end{algorithmic}
\end{algorithm*}

\subsection{Method of Adaptive Inexact Descent (MAID)}
Given the approximate hypergradient provided by \Cref{alg:piggyback_inexact}, we employ the adaptive backtracking line search scheme introduced in \cite{salehi2025adaptivelyinexactfirstordermethod} to find suitable upper-level step-sizes, and to adaptively adjust the tolerances $\epsilon^x, \epsilon^y, \delta^X, \delta^Y$ required for calculating an approximate hypergradient.
\rv[final]{We denote the $t$-th iterate of the linear learnable parameters by $K_t$.}
The following lemma gives a sufficient decrease condition which we will deploy as the line search rule for upper-level iterations.
\begin{lemma}\label{inexact_bt_lemma}
    \cite[\rv{Lemma 3.5}]{salehi2025adaptivelyinexactfirstordermethod} Let \Cref{ass_pigy_1} hold, $\lambda \in \mathbb{R}$, and suppose that the upper-level loss $\ell$ is convex. For each upper-level iteration $t= 0,1, \dots$, set $u_t \coloneqq (\tilde{x}(K_t), \tilde{y}(K_t))$, $\bar{\epsilon} = \max \{\epsilon^x_t, \epsilon^x_{t+1}, \epsilon^y_t, \epsilon^y_{t+1}\}$, and denoting 
    \begin{multline}\label{psi}
    \psi(\rvv{\alpha_t}) \coloneqq \ell(u_{t+1}) + \| \nabla \ell(u_{t+1}) \| \bar{\epsilon}\\
    + \frac{L_{1}+L_{2}}{2} \bar{\epsilon}^2 - \ell(u_{t}) + \| \nabla \ell(u_{t}) \| \bar{\epsilon} + \lambda \alpha_t \|z_t\|^2.
\end{multline} If the line search condition
$
        \psi(\alpha_t) \leq 0
$
     is satisfied, then the sufficient descent condition $\mathcal{L}(K_{t+1})-\mathcal{L}(K_{t})\leq -\lambda \alpha_t \|z_t\|^2$ holds.  
\end{lemma}

Considering the bilevel problem \eqref{piggy_bilevel} and utilizing \Cref{inexact_bt_lemma}, the adaptive backtracking line search with restart, outlined in \Cref{alg:MAID}, is employed to solve the upper-level problem \eqref{upperpiggy}. \rv{Note that, as detailed in the steps of \Cref{alg:MAID}, \Cref{inexact_bt_lemma} serves as an accuracy-dependent line search condition to adjust the required accuracy for progressing in the upper-level optimization iterations. When this condition is satisfied, it indicates that the accuracy is adequate for a sufficient decrease and suggests that the lower-level tolerance can be enlarged to reduce the computational cost of the next upper-level iteration. Conversely, if the condition does not hold and the maximum number of line search steps is reached, it indicates that the accuracy was insufficient. In this case, we proceed by shrinking the lower-level tolerance and restarting the computation of the lower-level solution, as well as rechecking the line search condition. }

\rv{Given that \Cref{lm:aposteriori_sc_saddle_adjoint} and \Cref{apos_piggy_theorem} provide a practical and computable framework, as outlined in \Cref{alg:piggyback_inexact}, for computing an inexact hypergradient required in \Cref{alg:MAID}, and that this aligns with the regularity assumptions needed in \Cref{inexact_bt_lemma} as well as the convergence requirements of \Cref{alg:MAID}, the convergence of our bilevel scheme follows from the theorem below.
}\rv{
\begin{theorem}\cite[Theorem 3.19]{salehi2025adaptivelyinexactfirstordermethod}
    Under \Cref{ass_pigy_1} and \ref{ass_pigy_2}, the iterates $K_t$ of \Cref{alg:MAID} satisfy
    $$\lim_{t\to\infty} \|\nabla\mathcal{L}(K_t)\| = 0.$$
\end{theorem}}
\begin{algorithm*}[h]
\caption{\cite[Algorithm 3.1]{salehi2025adaptivelyinexactfirstordermethod} Method of Adaptive Inexact Descent (MAID). Hyperparameters: $\underline{\rho} \in (0,1)$ and \rv{$\overline{\rho}>1$} control the reduction and increase of the step-size $\alpha_k$, respectively; $\underline{\nu} \in (0,1)$ and \rv{$\overline{\nu}>1$} govern the reduction and increase of accuracies $\delta_k$ and $\epsilon_k$; $\max_\text{BT} \in \mathbb{N}$ is the maximum number of backtracking iterations.}\label{alg:MAID}
\begin{algorithmic}[1]
\State Input $K_{0} \in \mathbb{R}^{C\times N}$, accuracies $\epsilon^x_0, \epsilon^y_0, \delta^X_0, \delta^Y_0>0$, step-size $\alpha_{0} >0$.
\For{$t=0, 1, \dots$}
\For{$j = \max_\text{BT}, \max_\text{BT}+1, \dots$ }\label{bt_loop}
\State{$z_t \leftarrow$ I\footnotesize{NEXACT}\normalsize P\footnotesize{IGGYBACK}\normalsize($K_t, \epsilon^x_t,\epsilon^y_t, \delta^X_t, \delta^Y_t$)} \label{updated_descend_direction}
\For{$i=0,1,\dots,j-1$}\label{inner_loop}
\If{inexact sufficient decrease $\psi(\alpha_t) \leq 0$ holds}\Comment{Lemma \ref{inexact_bt_lemma}}
\State{Go to line \ref{gd_update_step}}\Comment{Backtracking Successful}
\EndIf
\State{$\alpha_{t} \leftarrow \underline{\rho}\alpha_t$}
\Comment{Adjust the starting step-size}
\EndFor
\State $\epsilon^x_t, \epsilon^y_t \leftarrow \underline{\nu} \epsilon^x_t, \underline{\nu}\epsilon^y_t \label{BT_decrease} $  \Comment{Backtracking Failed and needs higher accuracy}
\State $\delta^X_t, \delta^Y_t \leftarrow \underline{\nu} \delta^X_t, \underline{\nu}\delta^Y_t$  \label{BT_decrease_delta}
\EndFor
\State{$K_{t+1} \leftarrow K_t - \alpha_{t} z_t$}\Comment{Gradient descent update}\label{gd_update_step}
\State{$\epsilon^x_{t+1}, \epsilon^y_{t+1} \leftarrow \overline{\nu} \epsilon^x_{t}, \overline{\nu} \epsilon^y_{t}$ }\label{increase_epsilon_k}\Comment{Increasing $\epsilon_t$}
\State{$\delta^X_{t+1}, \delta^Y_{t+1} \leftarrow \overline{\nu} \delta^X_{t}, \overline{\nu} \delta^Y_{t}$ }\label{increase_delta_k}\Comment{Increasing $\delta_t$}
\State{$\alpha_{t+1} \leftarrow \overline{\rho} \alpha_{t}$}\label{increase_beta_k}\Comment{Increasing $\alpha_t$}

\EndFor
\end{algorithmic}
\end{algorithm*}
\section{Numerical Experiments}
\label{sec:experiments}
We now provide numerical experiments to support the theoretical contributions of this work. Section~\ref{sec:experiments:tv} focuses on the problem of learning an improved discretization for the total variation (TV). Using the adaptive algorithm presented in Algorithm~\ref{alg:MAID}, we demonstrate its potential advantages compared to the native piggyback method without adaptive step-sizes and tolerances, which relies on careful tuning of the learning rate $\alpha$ and accuracy parameters $\delta$ and $\epsilon$. Subsequently, we take this a step further to demonstrate the applicability of our algorithm in bilevel settings where the piggyback algorithm is used to learn more expressive regularizers. More precisely, in Section~\ref{sec:experiments:icnn}, an ICNN-based regularizer is learned using the proposed algorithm for CT reconstruction, yielding competitive results. The accompanying source code will be made available upon publication.

\subsection{Learning the Discretization of Total Variation}\label{sec:experiments:tv}
In the work of~\cite{condat2017discrete}, the need for a new discretization for the total variation was identified. The aim was to enforce common pixel grid locations of the employed vector fields which in turn promotes rotational invariance that is not granted with the standard discretization. This was later extended to a general (consistent) setting using learnable interpolation filters for the total variation and its second-order extension~\cite{AntoninChambolle2021,bogensperger2023learned}. This is treated in the setting of standard linear inverse problems, where an image $x\in\mathbb{R}^{M\times N}$ is sought to be reconstructed. 
\rv[final]{In this experiment, we consider the task of image denoising. Given a noisy observation $u$, we consider the 1-strongly convex $L^2$ data fidelity $\tfrac{1}{2}\|x-u\|^2_2$.}

Following~\cite{condat2017discrete,AntoninChambolle2021}, which suggests to take the constraint on the interpolated dual variable $Ky$ instead of the dual variable $y$. 
The dual form of total variation can be represented as:
\begin{equation}\label{eq:tv_dual}
    \mathrm{TV}(x) = \sup_y \langle Dx,y\rangle:\Vert Ky\Vert_{Z}^* \leq 1,
\end{equation}
where $\Vert \cdot\Vert_Z:= \Vert \cdot \Vert_{2,1,1}$ for $q \in \mathbb{R}^{n\times 2 \times M \times N}$, which is the absolute sum of its $n$ components of the 2-norm of its two components. Using the linear operator $K$ this allows to interpolate the components of $y\in \mathbb{R}^{2\times M\times N}$, which are located on staggered pixel grids due to the employed finite differences $D$. Note that Neumann boundary conditions are used here to obtain resulting pixel grids of the same spatial size. Note that the constraint essentially implies that the operator $K$ encompasses a sparse coding for $Dx$.
The corresponding primal problem of~\eqref{eq:tv_dual} is then given by
\begin{equation}\label{eq:tv_primal}
    \mathrm{TV}(x) = \min_{q:K^* q=Dx} \Vert q \Vert_{Z},
\end{equation}


Combining~\eqref{eq:tv_primal} with the \rv{$L^2$ data fidelity}, this yields the subsequent saddle-point problem with regularization parameter $\lambda \in \mathbb{R}^+$: 
\rv{
\begin{equation}\label{eq:tv_sp_standard}
\min_{x,q} \max_y \tfrac{1}{2}\|x-y\|^2_2 + \lambda \Vert q\Vert_Z + \langle Dx-K^* q,y\rangle.
\end{equation}
}
Given the exact solution $(\hat{x},\hat{y})$ of~\eqref{eq:tv_sp_standard}, its corresponding adjoint saddle-point problem is then given by
\rv{
\begin{multline}
    \min_{X,Q} \max_Y \langle DX - K^* Q, Y \rangle + \tfrac{1}{2} \langle X, X \rangle \\
    + \tfrac{1}{2} \langle \nabla^2 \lambda \| \hat{q} \|_{Z}, Q \rangle + \langle \nabla \ell(\hat{x}, x^*), X \rangle.
\end{multline}
}
\rv[final]{In~\eqref{eq:tv_sp_standard}, both the non-smooth absolute function $\Vert \cdot \Vert_{2,1,1}$ and the $L^2$ data term jointly constitute the original function $g$ defined in the primal objective in~\eqref{piggy_lower} of the lower-level problem. Note, however, that the non-smooth term $\Vert \cdot \Vert_{2,1,1}$ does not fully satisfy the regularity assumptions imposed on $g$ in~\eqref{piggy_lower}.}
While this would typically require smoothing, we observe that the primal-dual style differentiation approach in bilevel learning performs robustly even in this less regular setting, similar to the findings in~\cite{Bogensperger2022}. Therefore, we directly work with~\eqref{eq:tv_sp_standard}.

For a given set of corresponding ground truth images and corrupted observations $(x_i^*,u_i)_i$, the standard bilevel setting in~\eqref{piggy_bilevel} can be employed, which given the saddle-point structure of the underlying lower-level problem can be approached using Algorithm~\ref{alg:MAID}. 
The experiments were conducted for the inverse problem of image denoising for a randomly sampled train subset of the BSDS train data set~\cite{arbelaez2010contour} with 32 images of size $128 \times 128$. To generate the corrupted observations for the image denoising task, zero-mean additive Gaussian noise $\sim \mathcal{N}(0,\sigma^2)$ with $\sigma=25.5$ was used \rvv{for image intensities scaled between 0 and 255}. \rvv{To evaluate the learned filters, another set of 32 test images of the same size and with the noise level was generated.}

\begin{figure*}[t!]
\begin{center}
\begin{minipage}[c]{0.36\linewidth}
    \subfloat[\footnotesize Learned filters with adaptive setting (top) and non-adaptive (bottom).]{\label{fig:learned_tv_filters}\includegraphics[width=\textwidth]{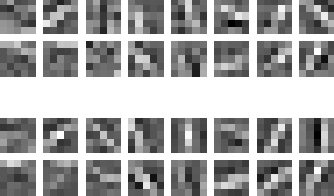}}
\end{minipage} \hspace{2pt}
\begin{minipage}[c]{0.58\linewidth}
    \subfloat[\centering \footnotesize Loss vs. computational budget of lower-level iterations]{\label{fig:tv_budget}\includegraphics[width=\textwidth]{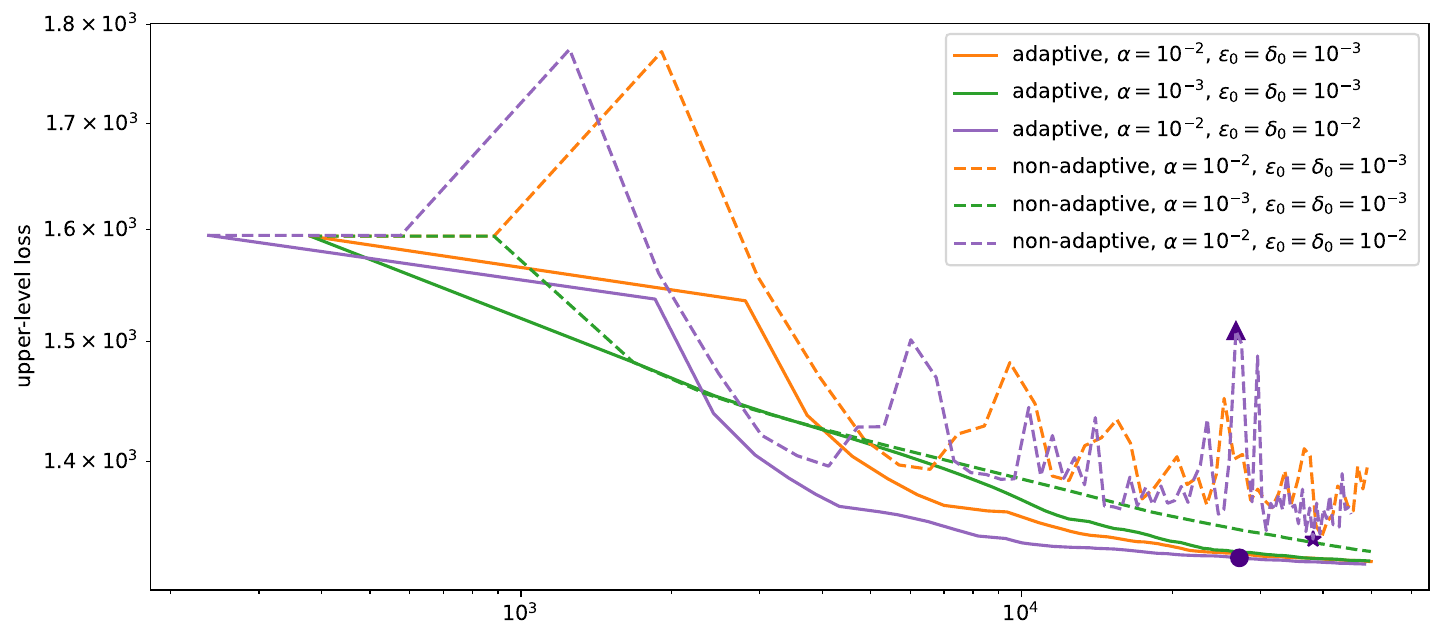}} 
\end{minipage}
\caption{Learned filters with adaptive and non-adaptive setting (a) and comparison of the upper-level loss depending on the computational budget which is comprised of the total number of upper- and lower-level iterations (b). For one experimental setting, the potential influence of the oscillating behaviour in the non-adaptive case is investigated by examining reconstruction results for different time-points of the oscillating curve (cf. \indigoTriangle vs. \indigoStar) and the corresponding results using Algorithm~\ref{alg:MAID} (\indigoCircle), see also Figure~\ref{fig:denoising_tv_oscillations}.}
\label{denoising_tv_budget}
\end{center}
\end{figure*}

The linear operator $K$ was heuristically chosen to consist of $n=8$ filters with filter kernels $5 \times 5$. \rvv{Across all experiments, filter weights were initialized with He‐uniform initialization~\cite{he2015delving} and a constant random seed.} The parameters in Algorithm~\ref{alg:MAID} were set to $\underline{\nu}=0.5$, $\bar{\nu}=1.05$, $\underline{\rho}=0.5$, $\bar{\rho}=\tfrac{10}{9}$ following~\cite{salehi2025adaptivelyinexactfirstordermethod}. For comparison between the adaptive and non-adaptive settings, where the latter corresponding to the native inexact version of piggyback, the initial values for $\epsilon_0$, $\delta_0$ and $\alpha_0$ were kept the same across both cases. However, in the non-adaptive version, these parameters were fixed throughout the algorithm. Figure~\ref{fig:learned_tv_filters} shows the learned filters with the adaptive Algorithm~\ref{alg:MAID} and without for $\epsilon_0=10^{-3}$, $\delta_0=10^{-3}$ and $\alpha_0=10^{-2}$, showing that both settings seem to capture very related orientations. While clearly the quantitatively learned filters are very similar, it is of higher interest to investigate the resulting implications on the used computational budget in terms of required lower-level iterations.  
Therefore, Figure~\ref{fig:tv_budget} shows the resulting upper-level loss depending on the required total computational budget in all standard and adjoint primal-dual algorithms for different experimental settings. The benefit of adaptively choosing the step-size can clearly be seen as the resulting upper-level loss decreases smoothly with increasing computational budget in the lower-level solver. 
\rv[final]{On the other hand, the dashed curves show scenarios with the same initial accuracies and learning rate $\alpha$ as in the adaptive setting, but these parameters are now kept constant throughout the learning algorithm.}
Depending on the initial learning rate, the upper-level loss even starts oscillating after an initial phase of decrease. 

Qualitative results can be seen in an example in Figure~\ref{figure:tv_disc_figure}, which shows reconstructions with standard TV and the learned discretizations with and without using adaptivity. Moreover, for the learned discretization filters, a threshold of $10^5$ lower-level iterations was used. The visualised results imply that for a limited budget there can even be visible differences in the resulting reconstructions. 

Finally, the potential ``worst-case'' scenario of the oscillating effect is demonstrated, which can arise if the learning rate is too high. In the exemplary setting of $\alpha=10^{-2}$ and $\epsilon=\delta=10^{-2}$ in Figure~\ref{fig:tv_budget} (dashed indigo curve), two points during the final learning phase were selected; one at the high-end of the oscillating upper-level loss curve (\indigoTriangle) and the other one at the low-end (\indigoStar). Using the learned filters for both settings yielded reconstruction results on the employed test image of $26.44$ dB and $26.70$ dB, respectively. This showcases that possible discrepancies with regards to reconstruction results depend on the stage of the oscillating upper-level loss. For comparison, the corresponding time-point for the same computational budget as in (\indigoTriangle) was also retrieved in the adaptive setting (\indigoCircle), which led to a reconstruction result of $26.88$ dB for the same test image. Moreover, to support these quantitative differences in the resulting reconstructions, this is supported qualitatively in Figure~\ref{fig:denoising_tv_oscillations}, where indeed a comparison even shows minor differences in the reconstructed images, emphasizing the potential benefits of using adaptive backtracking as proposed in Algorithm~\ref{alg:MAID} for this setting. Note that the non-adaptive setting in Figure~\ref{figure:tv_disc_figure} that was obtained for a budget of $10^5$ lower-level iterations shows quantitative reconstruction results somewhere in-between the two chosen extremes of the oscillating curve. 

\figDoubleColumn{0.99}{0 0 0 0}{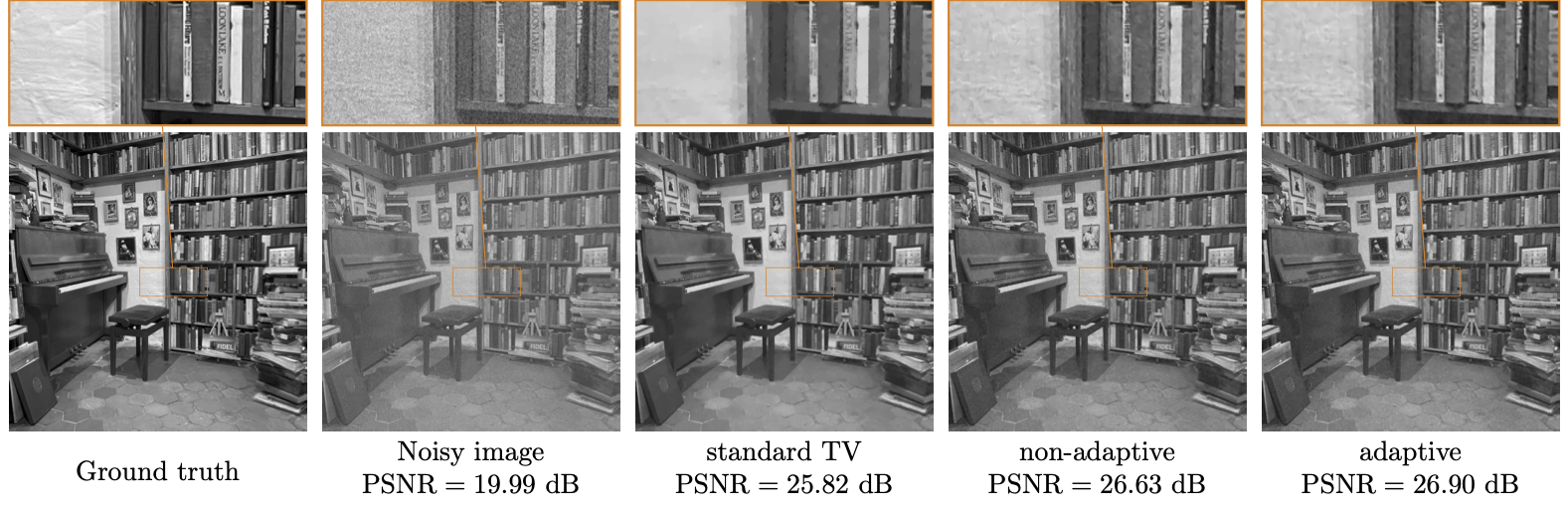}{tv_disc_figure}{TV denoising with standard TV and learned discretization filters showing the influence of using the adaptive learning setting at a fixed computational budget of $10^5$ total lower-level iterations. Best viewed on screen.}

\begin{figure}[tbhp]
\centering 
\includegraphics[width=\linewidth]{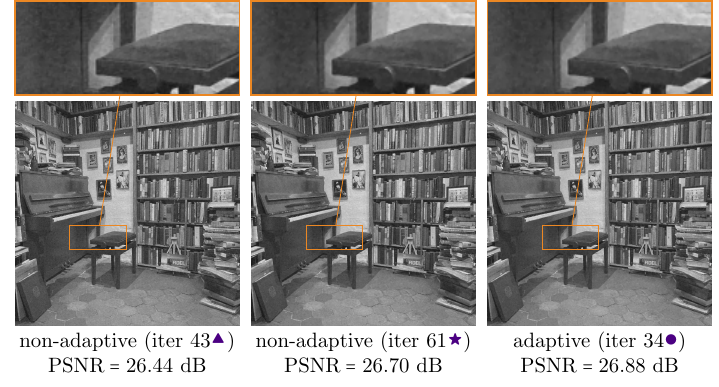}
\caption{TV denoising with the native inexact learning algorithm (non-adaptive setting) at two different time-points during the oscillating upper-level loss curve for initial parameters $\alpha=10^{-2}$ and $\epsilon=\delta=10^{-2}$: choosing a point at the high-end of the loss curve (\indigoTriangle) vs. the low-end (\indigoStar) can potentially have a significant increase in the resulting reconstruction. For completeness, the reconstruction from the corresponding time-point during the adaptive setting (\indigoCircle) is also shown. Best viewed on screen.}
\label{fig:denoising_tv_oscillations}
\end{figure}

\subsection{Learning Convex Regularizers} \label{sec:experiments:icnn}
In this section, we study the use of our algorithm within a bilevel learning framework to learn more expressive regularizers for imaging tasks. To illustrate this, we consider two convex data-driven regularizers based on input-convex neural networks (ICNNs). The first example is an ICNN which consists of two convolution layers, that is $R_{\theta}(x)=\psi_{w}(Wz)$ with $z=\psi_{w}(Vx)$. Here $V,W$ denote learnable convolution layers whose weights constitute the linear operator $K$, and $\psi_{w}$ is a smooth activation function defined as:
\begin{equation}
    \psi_w(x)=\begin{cases}
        0 &\text{ if }x\leq0,\\
        \frac{x^2}{2w}, &\text{ if }0<x<w,\\
        x-\frac{w}{2}, &\text{ otherwise},
    \end{cases}
\end{equation}
with smoothing parameter $w=0.01$. Following \cite{amos2017input}, we assume that $W$ is non-negative to ensure the convexity of $R_{\theta}$.

For the image reconstruction experiments, we focus on a sparse-view CT reconstruction task. The training and testing data sets consist of human abdominal CT scans provided by the Mayo Clinic Low-Dose CT Grand Challenge data set~\cite{mccollough2016tu}. To simulate measurements, a parallel beam geometry with 200 projection angles is employed, with additional Gaussian noise of standard deviation \( \sigma = 2 \) added. The training data set consists of 25 randomly sampled images from the original data set, each of size \( 512 \times 512 \).

Following the primal-dual framework introduced in \cite{wong2024primal}, we consider the following lower-level problem:
\begin{equation}\label{eq:icnn}
    \min_{x,z}\frac{1}{2}\|Ax-u\|^2+\frac{\mu_g}{2}\|x\|^2+\delta_C(Vx,z)+\gamma\psi_{w}(Wz),
\end{equation}
where $C:=\{(p,q)|\psi_w(p)\leq q\}$, and $\delta_C$ denotes the indicator function of $C$.
Here the linear operator $A$ denotes the X-ray transform with the prescribed geometry of the measurement setting and $u$ is the sinogram, representing the measured projections. \rv[final]{To ensure strong convexity of the primal function $g$, we introduce a quadratic regularization term weighted by $\mu_g$. In practice, we select a small value (\textit{e.g.}, $\mu_g = 10^{-8}$) to satisfy the strong convexity requirement without notably influencing the obtained results.}
For notational convenience, let $f_2(p,q)=\delta_C(p,q),f_3(v)=\gamma\psi_w(v)$. We then dualize the data fidelity  in~\eqref{eq:icnn} which gives rise to the saddle-point problem:
\begin{multline}\label{eq:icnn_spi}
       \min_{x,z}\max_{y_1,y_2,y_3}\langle Ax, y_1 \rangle + \frac{\mu_g}{2} \| x \|_2^2 
    - \frac{1}{2} \| y_1 + u \|_2^2\\ + \frac{1}{2} \|u\|_2^2 + \langle Vx,y_{2,1}\rangle + \langle z,y_{2,2}\rangle - f_2^*(y_2) \\+ \langle Wz, y_3 \rangle - f_3^*(y_3),
\end{multline} 
where $y_2=\{y_{2,1},y_{2,2}\}$. 
Additionally, the corresponding adjoint problem can be derived from this formulation, which is necessary for solving the corresponding adjoint saddle-point problem in our primal-dual style differentiation bilevel framework in Algorithm~\ref{alg:MAID}.

For the second example, we note that by choosing only one layer for the ICNN, we can recover the well-known Fields of Experts (FoE) objective~\cite{chen2014insights,roth2005fields}. The regularizer again consists of linear filters and the smoothed activation function $\psi_w$, which are once more combined with the convex data fidelity, yielding the following:
\begin{equation}\label{eq:foe_min}
\min_{x \in \mathcal{X}} \frac{1}{2} \Vert Ax - u\Vert_2^2+\frac{\mu_g}{2} \Vert x \Vert_2^2 + \gamma  \psi_w (Kx).
\end{equation} 
Note that this closely follows the primal objective of the ICNN with two layers in~\eqref{eq:icnn}, but the second variable $z$ that is minimized over can be omitted here due to using only one layer of ICNN.
Again, we dualize the non-smooth function as well as the data fidelity, which results in the saddle-point problem
\begin{multline}\label{eq:foe_sp}
    \min_{x} \max_{y_1, y_2} 
    \langle Ax, y_1 \rangle + \frac{\mu_g}{2} \| x \|_2^2 \\
    - \frac{1}{2} \| y_1 - u \|_2^2 + \langle Kx, y_2 \rangle - f^*(y_2).
\end{multline}

\begin{figure*}[t]
\centering
\includegraphics[width=\linewidth]{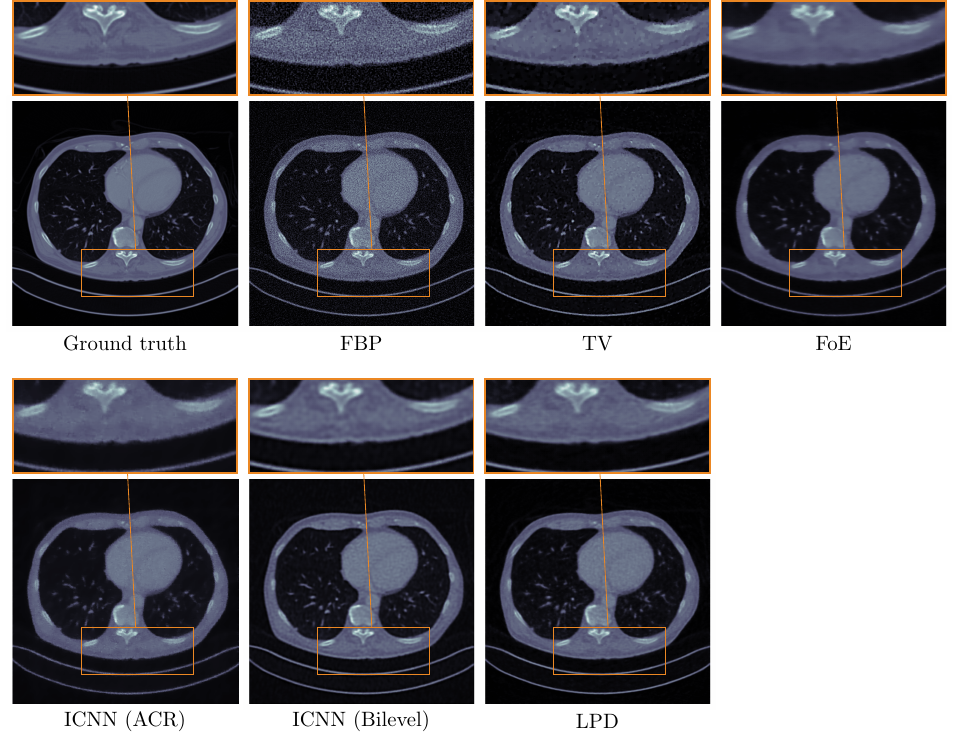}
\caption{Qualitative results of test image CT reconstructions in the sparse-view setting. FoE and ICNN (Bilevel) results were obtained using the learned filters with the proposed algorithm. Corresponding quantitative results can be found in Table~\ref{tab:ct_sparse}. Best viewed on screen.}
    \label{fig:ct_sparse_combined}
\end{figure*}

For the experiments, the FoE regularizer consists of $16$ linear filters with filter kernels of size $5 \times 5$ (denoted by the linear operator $K$). Meanwhile, for ICNN the first layer consists of 16 filters of size $5 \times 5$ and the second layer of $4\times 16$ filters of size $5 \times 5$.
While filtered backprojection (FBP) is the baseline reconstruction method in CT, we also include comparisons with other methods to ensure a comprehensive evaluation. Specifically, we compare the performance of Algorithm~\ref{alg:MAID} for the previously described 2-layer ICNN in~\eqref{eq:icnn_spi} trained in a bilevel manner against its adversarially trained counterpart (ACR). Additionally, we include comparisons with the FoE regularizer (i.e., the 1-layer ICNN in~\eqref{eq:foe_sp}), total variation (TV), and the data-driven learned primal-dual method (LPD) \cite{adler2018learned}. The specific learning settings are described in the following paragraph.

For the 2-layer ICNN, the parameters in \Cref{alg:MAID} are set to $\underline{\nu}=0.5$, $\bar{\nu}=1.05$, $\underline{\rho}=0.1$, $\bar{\rho}=1.1$. The initial parameters are set to $\alpha_0=10^{-4}$, $\epsilon_0=20$, $\delta_0=4$. 
The ACR model is trained using the framework described in \cite{mukherjee2020learned}. During training, the gradient penalty is set to 10, and the Adam optimizer is employed with a learning rate of $2\cdot10^{-5}$. The model is trained for 20 epochs on the data set outlined earlier.
For the FoE regularizer, similar as in Section~\ref{sec:experiments:tv}, the parameters in Algorithm~\ref{alg:MAID} were set to $\underline{\nu}=0.5$, $\bar{\nu}=1.05$, $\underline{\rho}=0.5$, $\bar{\rho}=\tfrac{10}{9}$. The initial parameters for $\epsilon_0$, $\delta_0$ and $\alpha_0$ were each set to $10$.
For the TV regularizer, we utilize the implementation provided by the Deep Inverse library\footnote{\url{https://deepinv.github.io/deepinv/}}. We perform 500 iterations of proximal gradient descent with a step-size of $10^{-3}$ and a hand-tuned regularization parameter of $2$. 
For LPD, we use the implementation provided in the publicly available\footnote{\url{https://deepinv.github.io/deepinv/auto_examples/unfolded/demo_learned_primal_dual.html}}. For our experiments, we employ $2$ primal-dual layers and use the AdamW optimizer \cite{loshchilov2017decoupled} implemented in PyTorch \cite{paszke2019pytorch}, with a learning rate of $10^{-3}$. Training is conducted for 30 epochs on the previously described data set.

Figure~\ref{fig:ct_sparse_combined} shows the qualitative results of reconstructing the test image in the sparse-view measurement scenario for all the above described methods as well as the corresponding ground truth image. 
Visually, the reconstruction results of the 2-layer ICNN trained using our proposed bilevel learning approach, as well as those from the learned primal-dual (LPD) method, exhibit high-quality reconstructions. They avoid both oversmoothed artifacts and overly fine-grained structures, which can often occur with other reconstruction methods.
To support these results, we also show quantitative results \rvv{for the test image} in Table~\ref{tab:ct_sparse}, which contains peak signal-to-noise ratio (PSNR) and structural similarity index measure (SSIM) for all methods on the employed test phantom. Note that this underlines the qualitative impressions from Figure~\ref{fig:ct_sparse_combined}. While the best scores are obtained using LPD, also the ICNN regularizer trained in our bilevel settings shows very competitive results. To some extent this can also be explained the number of learnable parameters which is in a ten-fold increase for LPD. 
On the other hand, the Fields of Experts regularizer that was trained using Algorithm~\ref{alg:MAID} shows remarkable reconstruction results given its reduced complexity and limited number of learnable parameters. 
Even more interestingly, using the exact same architecture for the 2-layer ICNNs, our bilevel learning framework manages to outperform ACR by a large margin. Since they have identical learnable parameters, this shows very well the efficiency of the employed bilevel learning framework using primal-dual style differentiation. 
Overall the experiments show competitive results of FoE and ICNN-based regularizers trained using Algorithm~\ref{alg:MAID}. 
\rv[final]{Future work could investigate the performance of the proposed method in more challenging scenarios, such as limited-angle measurement settings. Furthermore, the entire framework can be directly extended to other linear inverse problems with data-driven and potentially non-smooth regularizers.}

\begin{table}[htbp]
    \caption{PSNR and SSIM values of test reconstructions shown in Figure~\ref{fig:ct_sparse_combined} for different methods for sparse-view CT reconstruction \rvv{on the test image}. Best score for each metric in \textbf{bold}, second-best \underline{underlined}.}
    \centering
    \begin{tabular}{lccc}
    \hline
         Method &  PSNR $\uparrow$ & SSIM $\uparrow$ & $\#$ Parameters\\
         \hline \hline
        FBP & 21.07 & 0.23 & -\\
        TV & 28.68 & 0.74  & 1\\
        FoE~\cite{chen2014insights} & 30.69& \underline{0.86} & 400\\
        ICNN (ACR)~\cite{mukherjee2020learned} & 29.32& 0.79 & 2000\\
        ICNN (Bilevel) & \underline{31.43} &\underline{0.86}& 2000\\ 
        LPD~\cite{adler2018learned} & \textbf{34.18} &\textbf{0.88} & 50397\\ 
        \hline
    \end{tabular}
    \label{tab:ct_sparse}
\end{table}


\section{Conclusions}
\label{sec:conclusions}
In this work, we proposed a-posteriori error bounds for hypergradients computed via primal-dual style differentiation. By integrating this error analysis and control into an adaptive inexact method, we developed a robust and cost-efficient training framework leveraging primal-dual style differentiation. Specifically, we introduced a bilevel learning framework for training regularizers in variational image reconstruction, utilizing primal-dual style differentiation. In the scenario of learning improved discretizations of the total variation, we compared our adaptive learning framework to the non-adaptive case, highlighting its superior performance, especially with a constrained computational budget and its robust reconstruction results. Additionally, by reformulating well-known data-adaptive convex regularizers, we demonstrated the effectiveness of inexact primal-dual style differentiation in training these regularizers, achieving superior performance compared to state-of-the-art methods.

For future work, extending our analysis to non-smooth variational problems would be a valuable direction. Furthermore, adapting this framework to stochastic settings could offer significant speed-ups and enable application to larger data sets. \rv{A comparison between our framework and single-level methods, particularly \cite{singleloopValkonen}, which shares a similar type of lower-level structure, could be a promising direction for future work. This comparison could be further extended to fully first-order methods \cite{bome,fullyFirstOrderStochastic,penaltyMethodsNonconvexBO}, providing deeper insights into their relative advantages and limitations.} Finally, exploring the connections between this framework and other inexact methods, such as implicit function theorem-based approaches with linear solvers and automatic differentiation, represents another promising avenue for research. 

\bmhead{Acknowledgements}
The work of Mohammad Sadegh Salehi was supported by a scholarship from the EPSRC Centre for Doctoral Training in Statistical Applied Mathematics at Bath (SAMBa), under the project EP/S022945/1. Matthias J. Ehrhardt acknowledges support from the EPSRC (EP/S026045/1, EP/T026693/1, EP/V026259/1). Hok Shing Wong acknowledges support from the project EP/V026259/1.









\bibliography{sn-bibliography}

\end{document}